\documentclass[12pt]{amsart}
\usepackage {amssymb,latexsym}
\usepackage [all]{xy}
\usepackage{nohyperref}
\usepackage[pageref]{backref}

\newtheorem{theorem}{Theorem}[section]
\newtheorem{lemma}[theorem]{Lemma}
\newtheorem {proposition}[theorem]{Proposition}
\newtheorem {corollary} [theorem] {Corollary}

\theoremstyle{definition}
\newtheorem{definition}[theorem]{Definition}
\newtheorem{example}[theorem]{Example}

\newtheorem{remark}[theorem]{Remark}
\newtheorem{assumption}[theorem]{Assumption}
\numberwithin {equation}{section}

\renewcommand*{\backref}[1]{}
\renewcommand*{\backrefalt}[4]{%
    \ifcase #1 (Not cited.)%
    \or        (Cited on page~#2.)%
    \else      (Cited on pages~#2.)%
    \fi
    }

\def\cs{{$C^{\ast}$}}

\def\q{{\mathbb{Q}}}
\def\z{{\mathbb{Z}}}
\def\n{{\mathbb{N}}}
\def\c{{\mathbb{C}}}
\def\r{{\mathbb{R}}}

\def\csr{{C^{\ast}_r}}

\def\H{{\mathcal{H}}}

\def\gh{{(G,H)}}
\def\ghr{{(G_r,H_r)}}

\def\gb{{\overline{G}}}
\def\hb{{\overline{H}}}
\def\ghb{{(\gb,\hb)}}

\def\la{{\lambda}}

\def\th{{\theta}}
\def\ff{{\varphi}}
\def\d{{\delta}}
\def\D{{\Delta}}

\def\ba{{\backslash}}
\def\inv{{^{-1}}}
\def\s{{^{\ast}}}
\def\st{{^{\circledast}}}
\def\hg{{H\backslash G}}

\begin{document}
\title{Locally compact Hecke pairs}

\author{Vahid Shirbisheh}
\email{shirbisheh@gmail.com}



\begin{abstract}
    We introduce an extended setting to study pairs $\gh$ of locally compact groups which admit a regular representation on $L^2(\hg)$ and therefore a \cs-algebra, called Hecke \cs-algebra. These pairs are called locally compact Hecke pairs and mainly consist of locally compact groups and their compact subgroups, or cocompact subgroups, or open Hecke subgroups. We clarify relationship of our generalized formulation with the discrete case and obtain new results for discrete Hecke pairs too. For instance, we show that the left regular representations associated to discrete Hecke pairs are bounded homomorphisms. We also study the unimodularity of discrete Hecke pairs and introduce a condition, significantly weaker than commutativity, implying unimodularity.
\end{abstract}
\maketitle

\section {Introduction}
\label{sec:INTRO}
Hecke \cs-algebras were introduced by Jean-Beno\^{\i}t Bost and Alain Connes in \cite{bc}, in order to study the class field theory of the field $\q$ of rational numbers by means of quantum statistical mechanics. Since then, many examples of Hecke pairs $\gh$ have appeared in the related literature. In these Hecke pairs either $H$ is a subgroup of a discrete group $G$ or $H$ is a compact open subgroup of a locally compact group $G$. In both cases the necessary and sufficient condition to define a regular representation and subsequently to associate a \cs-algebra to the pair $\gh$ is that $H$ should be commensurable with all its conjugates. In this setting the homogeneous space $\hg$ associated to the pair $\gh$ is always a discrete space, and so we refer to such pairs as discrete Hecke pairs. The discreteness of the homogeneous space $\hg$ facilitates many constructions and results. However, this setting misses many interesting pairs of locally compact groups and their subgroups which have been studied in other fields, for instance $(SL_2(\r), SO_2(\r))$ and $(SL_2(\r), SL_2(\z))$. We shall see that these pairs also admit  \cs-algebraic representations similar to the Hecke \cs-algebras previously defined in the literature. On the other hand, it was observed in \cite{tzanev, klq} that the theory of locally compact groups provides a convenient setting for studying discrete Hecke pairs and their associated Hecke \cs-algebras. Another motivation for studying locally compact Hecke pairs comes from our results in the present paper, and in our another paper \cite{s4} where we apply our generalized setting based on locally compact groups to study length functions on Hecke pairs and property (RD) for locally compact Hecke pairs. Therefore our main aim in this paper is to extend the definition of Hecke \cs-algebras for several classes of pairs $\gh$ of locally compact groups whose homogeneous spaces $\hg$ are not necessarily discrete.

In this paper, a pair $\gh$ consisting of a locally compact group $G$ and a closed subgroup $H$ of $G$ is simply called a {\it pair}. We carefully examine various steps of associating a \cs-algebra to a ``suitable'' pair. Our primary concern is to determine classes of pairs such that for every pair $\gh$ in them, there exists a left regular representation from the Hecke algebra $\H\gh$ associated to $\gh$ into the \cs-algebra of bounded operators on the Hilbert space $L^2(\hg)$. Besides discrete Hecke pairs, we show that there are at least two other important classes of pairs which admit such a construction, and therefore their Hecke algebras are suitable for \cs-algebraic formulation. The first class consists of pairs $\gh$ in which $H$ is a compact subgroup of $G$. In this case, not only we prove the existence of a well defined left regular representation, but we also show that it is a bounded homomorphism, see Theorem \ref{thm:Hcompact}. Although this latter statement is not a necessary step to define Hecke \cs-algebra, it is a useful feature in application, due to the parallelism between reduced Hecke \cs-algebras and reduced group \cs-algebras. Furthermore, by applying a certain topologization process of discrete Hecke pairs named the Schlichting completion, this statement is applied to prove the same boundedness statement for the left regular representations associated with discrete Hecke pairs. This is particularly important in the study of property (RD) for Hecke pairs, for instance it is applied to show that every locally compact Hecke pair with polynomial growth has property (RD), see Section 4 of \cite{s4} for details. The second class of locally compact Hecke pairs capable of a \cs-algebraic representation consists of pairs $\gh$ in which $H$ is a cocompact subgroup of $G$, see Theorem \ref{thm:cocompact}.

To every discrete Hecke pair $\gh$, a group homomorphism $\D_\gh:G\to \q^+$ is associated, which we name it ``the relative modular function''. In a sense, it is a generalization of the modular function of a locally compact group. Therefore we define the involution on the Hecke algebra associated with $\gh$ using its relative modular function. Then it follows that the associated left regular representation is an involutive homomorphism if the relative modular function is the constant function 1. Moreover, in Section 3 of \cite{s4}, we have proved that a necessary condition for a discrete Hecke pair $\gh$ to possess property (RD) is that the relative modular function must be the trivial homomorphism, i.e. $\D_\gh=1$. It is in line with a theorem of R. Ji and L.B. Schweitzer in \cite{JiSch} which asserts that if a locally compact group $G$ has property (RD), then it must be unimodular. Similar conclusion can be made about property (T) of discrete Hecke pairs, see Proposition \ref{prop:PTunimod}. These are some of the evidences that even discrete Hecke pairs have some common features with locally compact groups. Motivated by these observations, we discuss various algebraic and analytic criteria to determine when the relative modular function of a discrete Hecke pair $\gh$ is trivial. For instance, it is shown that if the Hecke algebra associated to $\gh$ satisfies a certain condition, weaker than commutativity, then the relative modular function is trivial, see Proposition \ref{prop:communimod}.

Following the works of G. Schlichting in \cite{sch1, sch2}, K. Tzanev, S. Kaliszewski, Magnus B. Landstad, and John Quigg in \cite{tzanev, klq}, constructed a certain procedure for densely embedding a discrete (reduced) Hecke pair $\gh$ into a Hecke pair $\ghb$, where $\gb$ is a totally disconnected locally compact group and $\hb$ is a compact open subgroup of $\gb$. This technique is called the Schlichting completion of Hecke pairs and it is our main tool to transfer several results from locally compact groups and Hecke pairs into the setting of discrete Hecke pairs.

The works presented in this paper and its companion \cite{s4} not only contribute to the theory of Hecke \cs-algebras and property (RD) for Hecke pairs, but they also serve as the prototypes of ideas and methods which relate the study of Hecke pairs with the area of locally compact groups. We hope that this paper provides sufficient preparations and clues for further developments in other aspects related to these areas.

\section {Locally compact Hecke pairs $\gh$}
\label{sec:LCHP}

Our main aim in this section is to define an extended setting to study Hecke pairs $\gh$ which admit left regular representations on $L^2(\hg)$. Here, our guiding principle in developing the theory of locally compact Hecke pairs is that it should coincide with the discrete case when the subgroup in the Hecke pair is an open Hecke subgroup. In this way, we will be able to transfer useful results and ideas from locally compact case into discrete case and vice versa. In order to define a general and consistent \cs-algebraic framework to study Hecke algebras, we need to define both an algebra and an involution on this algebra. These two steps are handled slightly differently in discrete and non-discrete cases.

\begin{definition}
\label{def:discretehp}
\begin{itemize}
\item [(i)]  Let $\gh$ be a pair of locally compact groups. It is called a {\it discrete pair}, if $H$ is open in $G$, otherwise it is called {\it non-discrete}.
\item [(ii)] A discrete pair $\gh$ is called a {\it discrete Hecke pair} if every double coset of $H$ in $G$ is a finite union of finitely many left cosets of $H$ in $G$. In this case, we also say that $H$ is a {\it Hecke subgroup of $G$}\footnote{An alternative name for Hecke subgroups appearing in the literature is ``almost normal subgroups'', see \cite{bc, klq}. However, we follow \cite{klq}. See also Example \ref{exa:Heckeoriginal} for the origin of the name ``Hecke subgroup''. Besides, we use the phrase ``almost normal'' for another notion, see Definition \ref{def:commen-nearly}(iii).}.
\item [(iii)] Given a discrete Hecke pair $\gh$, the vector space of all finite support complex functions on the set $G//H$ of double cosets of $H$ in $G$ is denoted by $\H\gh$ and the above condition allows us to define a {\it convolution product} on $\H\gh$ by
\begin{equation}
\label{eqn:Hcon1}
f_1\ast f_2 (HxH):=\sum_{Hy\in \hg} f_1(Hxy\inv H)f_2(HyH),
\end{equation}
for all $f_1, f_2\in \H\gh$ and $x\in G$.
\end{itemize}
\end{definition}
Given a discrete Hecke pair, the vector space $\H\gh$ with the above product is a complex algebra, which is usually called the Hecke algebra associated with the pair $\gh$, see \cite{hall, krieg} for detailed introductions to this type of Hecke algebras. The algebra $\H\gh$ also admits an involution defined by
\begin{equation}
\label{eqn:traditionalinv}
f\st (HxH):=\overline{f(Hx\inv H)}, \qquad \forall f\in \H\gh, x\in G.
\end{equation}
However, due to some technical reasons explained below, we have to consider a slightly different involution on $\H\gh$, see Definition \ref{def:invdiscrete}.

Using a convolution product similar to (\ref{eqn:Hcon1}), $\H\gh$ is endowed with a {\it left regular representation} as follows:
\begin{equation}
\label{eqn:lregrep-dis}
\la:\H\gh\to B(\ell^2(\hg)),
\end{equation}
\[
\quad \la(f)(\xi)(Hx):=(f\ast \xi)(Hx),
\]
for all $f\in \H\gh, \xi \in \ell^2(\hg)$ and $Hx\in \hg$. By definition, the {\it reduced Hecke \cs-algebra of the Hecke pair $\gh$} is the norm completion of the image of this representation. The class of these Hecke pairs also includes pairs $\gh$, where $G$ is a locally compact group and $H$ is a compact open subgroup of $G$. These Hecke pairs and their associated Hecke \cs-algebras have been studied extensively in the literature, see \cite{curtis, hall, klq, tzanev}.

Now we turn our attention to non-discrete pairs. As it is clear from the discrete case, we are not only interested in the Hecke algebra of a Hecke pair $\gh$, but also we need an explicit regular representation to embed this algebra inside the \cs-algebra of bounded operators on the Hilbert space $L^2(H\ba G)$, where the homogeneous space $H\ba G$ is equipped with an appropriate  measure. This requirement makes us to define the convolution product, $L^1$-norm and $L^2$-norm using integrations over the locally compact space of right cosets of $H$ in $G$, and therefore, we need to set up an appropriate measure theoretic framework. In what follows, we often follow the notations of \cite{folland-ha} for measure theoretic considerations.

In this paper, $\mu$ always denotes a right Haar measure on $G$, $\eta$ denotes a right Haar measure on $H$ and $\Delta_G$ (resp. $\Delta_H$) denotes the modular function of $G$ (resp. $H$). The vector space of all complex valued compact support continuous functions on a (locally compact) topological space $X$ is denoted by $C_c(X)$. There is an onto map $P:C_c(G)\to C_c(H\ba G)$ defined by
\[
Pf(Hx):=\int_H f(hx) d\eta (h), \qquad \forall f\in C_c(G), x\in G.
\]
It is well known that we have
\begin{equation}
\label{eqn:invariantmeasure}
\Delta_G |_H=\Delta_H
\end{equation}
if and only if there exists a right $G$-invariant Radon measure $\nu$ on $H\ba G$. In this case, $\nu$ is unique up to a positive multiple. By choosing the multiple suitably, we obtain {\it Weil's formula} for the decomposition of an integral on $G$ into a double integral on $H$ and $\hg$ as follows:
\begin{equation}
\label{eqn:invmeasure}
\int_G f(x)d\mu(x)= \int_{\hg } Pf(y) d\nu(y)=\int_{\hg } \int_H f(hy) d\eta(h) d\nu(y),
\end{equation}
for all $f\in C_c(G)$. Then we briefly say that {\it the triple $(\eta, \mu, \nu)$ satisfies Weil's formula}. One notes that $y$ in the above formula varies over a complete set of representatives of right cosets. In fact, we could have used the notation $Hy$ instead of $y$, and we sometimes use the right coset notation in order to clarify our computations.
\begin{assumption}
\label{assume:r-G-inv}
We restrict our study to those pairs $\gh$ whose homogeneous spaces $H\ba G$ possess right $G$-invariant Radon measures, denoted by $\nu$, satisfying Equality (\ref{eqn:invmeasure}).
\end{assumption}
Inspired by the discrete case, for a given non-discrete pair $\gh$, we define
\begin{equation}
\label{eqn:Hgh-nondis}
\H\gh:=\left\{f\in C_c(\hg); f(Hxh)=f(Hx),\, \forall\,  x\in G, h\in H \right\}.
\end{equation}
Every element of $\H\gh$ can be thought of as a function on the set of double cosets of $H$ in $G$. Therefore, for every $f\in \H \gh$ and $x\in G$, each of the expressions $f(HxH)$, $f(xH)$ and $f(x)$ has the same meaning as $f(Hx)$. The complex vector space $\H\gh$ is equipped with the following convolution product:
\begin{equation}
\label{eqn:Hgh-convolution-nondis}
f\ast g (Hx):=\int_{\hg} f(Hxy\inv H)g(Hy)d\nu(Hy),\quad \forall f,g \in \H\gh.
\end{equation}
For $f$ and $g$ as above, let $S_f$ and $S_g$ be the supports of $f$ and $g$, respectively. Let $\pi:G\to \hg$ be the natural quotient map. By Lemma 2.46 of \cite{folland-ha}, there are compact subsets $A_f,A_g\subseteq G$ such that $\pi(A_f)=S_f$ and $\pi(A_g)=S_g$. Since $A_f A_g$ is a compact subset of $G$, $\pi(A_fA_g)$ is compact, and one easily observes that $supp(f\ast g) \subseteq \pi(A_fA_g)$.  It is straightforward to check that $f\ast g$ is continuous. Finally, it follows from the right $G$-invariance of $\nu$ that  $f\ast g\in \H\gh$. By applying appropriate change of variables, the right $G$-invariance of $\nu$ and Fubini's theorem, it is straightforward to check that the above convolution product is associative. Thus Formula (\ref{eqn:Hgh-convolution-nondis}) defines a product on $\H\gh$. As a technical point, one notes that as long as a compact support function appears in the integrand of a double integral, Fubini's theorem can be applied to change the order of integration. In many ``straightforward'' computations omitted from our arguments, this point allows us to use Fubini's theorem. 
\subsection*{Defining involutions}
\label{subsec:involution}
In order to define appropriate involutions on the Hecke algebra $\H\gh$, again we have to treat discrete and non-discrete cases, separately. We also note the following remark:
\begin{remark}
\label{rem:inverse-involution}
    Since our formulation of Hecke pairs and Hecke algebras is based on right cosets and right Haar measures, in any formula that contains modular functions, we have to consider the inverse of the modular function. For instance, given a right Haar measure $\mu$ on a locally compact group $G$ with the modular function $\Delta$, we have
    \begin{equation*}
    \mu(xE)=\Delta(x^{-1})\mu(E),\qquad \forall x\in G, \quad E\subset G \text{ measurable}.
    \end{equation*}
    While the above formula, considering a left Haar measure $\theta$, would be $\theta(Ex)=\Delta(x)\theta(E)$. In what follows, a similar modification will be considered regarding discrete Hecke pairs.
\end{remark}
We proceed with some notations for the discrete case. For every discrete Hecke pair $\gh$, we define two functions $L, R:G\to \n$ by
\[
L(x):=[H:H_x]=|HxH/H|, \qquad R(x):= [H:H_{x\inv}] =|H\ba HxH|,
\]
for all $x\in G$, where $H_x:=H\cap xHx\inv$. The integer $L(x)$ (resp. $R(x)$) is the number of distinct left (resp. right) cosets appearing in the double coset $HxH$, and so we have $R(x)=L(x\inv)$.
\begin{definition}
\label{def:relativemod}
The {\it relative modular function of a discrete Hecke pair $\gh$} is the function $\Delta_\gh:G\to \q^+$ defined by
\[
\Delta_\gh(x):=\frac{L(x)}{R(x)}, \qquad \forall x\in G.
\]
The discrete Hecke pair $\gh$ is called {\it relatively unimodular} if $\Delta_\gh(x)=1$ for all $x\in G$.
\end{definition}
In fact, the function $\Delta_\gh$ is a group homomorphism whose kernel contains $H$, see for instance Proposition 2.2 of \cite{tzanev}. It follows that $\Delta_\gh$ is always a continuous bi-$H$-invariant function on $G$. Moreover, when $H$ is a compact open subgroup of $G$, it was noted in \cite{tzanev, klq} that
\begin{equation}
\label{eqn:relativeequality}
\Delta_\gh(x)=\Delta_G(x),\qquad \forall x\in G,
\end{equation}
see for instance Page 669 of \cite{klq} for a proof. We should mention that the above equality also appeared in Lemma 1 of G. Schlichting's paper \cite{sch1}. One also notes that the left hand side of the above equation is algebraic, thus as long as $H$ is a compact open subgroup of a locally compact group $G$, the above equation holds.
\begin{remark}
\label{rem:l1-norm}
    Although we defined $\H\gh$ as the algebra of finite support functions over the set $G//H$ of double cosets of $H$ in $G$, we define the $\ell^1$-norm of elements of $\H\gh$ as functions over the set $\hg$ of right cosets of $H$ in $G$, more precisely
    \begin{equation*}
    \|f\|_1:=\sum_{Hx\in \hg} |f(Hx)| = \sum_{HyH\in G//H} R(y)|f(HyH)|, \quad \forall f\in \H\gh.
    \end{equation*}
\end{remark}
It is important to note that the left regular representation $\la$, defined in (\ref{eqn:lregrep-dis}), is an involutive homomorphism with respect to the involution defined in (\ref{eqn:traditionalinv}). However, this involution does not necessarily preserve the $\ell^1$-norm of elements of $\H\gh$. Even worse, one can see that this involution is not always continuous with respect to the $\ell^1$-norm. Therefore we are not able to show directly that the left regular representation $\la$ is continuous. The following example illustrates these points:
\begin{example}
\label{exa:involutionnot}
Let $\gh$ be the Bost-Connes Hecke pair, that is
\[
G=\left\{ \left( \begin{array} {rr}1&b\\0&a \end{array}\right); a\in \q^+, b\in \q \right\}, \quad \text{and} \quad H=\left\{ \left( \begin{array} {rr}1&n\\0&1
\end{array}\right); n\in \z \right\}.
\]
For given $g=\left( \begin{array} {rr}1&b\\0&\frac{m}{n} \end{array}\right)\in G$, where $m$ and $n$ have no common prime factors, one easily computes $L(g)=n$ and $R(g)=m$, see 2.1.1.3 of \cite{hall}. Let $\chi_g$ denote the characteristic function of the double coset $HgH$ considered as an element of $\H\gh$. Then $\|\chi_g\|_1 = R(g)=m$ and  $\|\chi_g\st\|_1 = L(g)=n$. Therefore by replacing $g$ with the elements of the sequence $\left\{\left( \begin{array} {rr}1&0\\0&\frac{1}{n} \end{array}\right)\right\}_{n\in \n}$, we observe that the involution $\st$ on $\H\gh$ is not continuous in $\ell^1$-norm.
\end{example}
The problem is that the mapping $Hx\mapsto Hx\inv$ is not even a well defined change of variable in $\hg$, of course, unless $H$ is a normal subgroup of $G$. However, the map $HxH\mapsto Hx\inv H$ is a well defined bijection over the set of double cosets of $H$ in $G$ and so the involution $\st$ is well defined. To obtain a norm preserving involution which can be generalized to non-discrete case, we use the definition given in \cite{klq} (with an appropriate modification regarding Remark \ref{rem:inverse-involution}):
\begin{definition}
\label{def:invdiscrete}
For a discrete Hecke pair $\gh$, the following involution is defined on the algebra $\H\gh$:
\begin{equation}
\label{eqn:invdiscrete}
f\s (HxH):=\D_\gh(x) \overline{f(Hx\inv H)},\quad \forall f\in \H\gh, x\in G.
\end{equation}
\end{definition}
Given a discrete Hecke pair $\gh$ and $x\in G$, let $\chi_x$ denote the characteristic function of the double coset $HxH$. Then the following identity immediately follows from the above definition, see also Page 660 of \cite{klq}:
\begin{equation}
\label{eqn:invcharacteristic}
\chi_x^\ast=\Delta(x\inv)\chi_{x\inv},\quad \forall x\in G.
\end{equation}
Now, using this identity, one checks that the involution defined in Definition \ref{def:invdiscrete} preserves the $\ell^1$-norm, that is
\begin{equation*}
\|f^\ast\|_1=\|f\|_1,\quad \forall f\in \H\gh.
\end{equation*}
Now we turn our attention to non-discrete pairs. In the following example, we observe that the same definition for relative modular functions on non-discrete pairs is not possible:
\begin{example}
\label{exa:sl2r1} Let $G$ be the special linear group of degree 2, $SL_2(\r)$, and let $H$ be its compact subgroup $SO_2(\r)=\left\{ \left( \begin{array}{cc} cos \th  & -sin\th \\ sin\th & cos\th \end{array} \right); -\pi<\th\leq \pi \right\}$. For a non-zero real number $t$, set $x:=\left(\begin{array}{cc} e^t & 0 \\ 0 & e^{-t} \end{array} \right)$. Then one computes that $H_x=\{ I, -I\}$, where $I$ is the $2\times 2$ identity matrix. Therefore $[H:H_x]=\infty$. This shows that functions $R$, $L$ and $\D_\gh$ are not well defined for the pair $\gh$, but we shall see that the pair $\gh$ admits a left regular representation, and so it will be considered as a (non-discrete) Hecke pair in our generalized setting.
\end{example}

Due to the lack of relative modular functions on non-discrete Hecke pairs, Formula (\ref{eqn:invdiscrete}) cannot be adopted for the non-discrete case. However, Equality (\ref{eqn:relativeequality}) suggests that we use the modular function of $G$ in the non-discrete case in lieu of the relative modular function in the discrete case. In order to this work, we have to impose another restriction:
\begin{assumption}
\label{assume:H-unimod}
When $\gh$ is a non-discrete pair, we assume $H$ is unimodular.
\end{assumption}
Thus, regarding the above assumption and Assumption \ref{assume:r-G-inv}, for a non-discrete Hecke pair $\gh$, we always assume $\D_G|_H=\D_H=1$. When there is no risk of confusion, we can ignore $\D_H$ and denote the modular function of $G$ simply by $\D$. It follows from the above assumption that $\D$ is a bi-$H$-invariant function on $G$, so the following definition is allowed:
\begin{definition}
\label{def:invo-nondis}
Assume $\gh$ is a non-discrete pair. The involution on the algebra $\H\gh$ is defined by
\begin{equation}
\label{eqn:inv-nondis}
f\s(Hx):= \D(x) \overline{f(Hx\inv )}, \qquad \forall f\in \H\gh,\, x\in G.
\end{equation}
\end{definition}
To see how the right $G$-invariance of $\nu$ is used to prove that this map is an involution on $\H\gh$, we show that it is an anti-homomorphism. For every $f,g\in \H\gh$ and $x\in G$, we compute
\[
(f\ast g)\s(Hx)= \D(x)\overline{f\ast g (Hx\inv)}=\D(x) \overline{\int_\hg f(x\inv y\inv)g(y) d\nu(y)}
\]
Set $Hz:=Hyx$. Due to the right $G$-invariance of $\nu$, this is a measure preserving change of variable over $\hg$. Thus we have
\begin{eqnarray*}
(f\ast g)\s(Hx)&=& \int_\hg \D(x) \overline{f(z\inv)}\ \overline{g((xz\inv)\inv)}d\nu(z)\\
&=& \int_\hg \left[\D(xz\inv) \overline{g((xz\inv)\inv)}\right] \left[\D(z)\overline{f\s (z\inv) }\right] d\nu(z) \\
&=& \int_\hg g\s (xz\inv)f\s (z) d\nu(z)=g\s \ast f\s(Hx).
\end{eqnarray*}
The proof of the rest of axioms of involutions are straightforward, and so are left to the reader. 
\subsection*{Defining Hecke pairs and Hecke algebras}
\label{subsec:Hecke-pairs-algebras}
Now that we have defined appropriate involutions for both discrete and non-discrete cases, we are ready to define Hecke algebras:
\begin{definition}
\label{def:heckealg-dis-nondis}
\begin{itemize}
\item [(i)] When $\gh$ is a discrete Hecke pair, the algebra $\H\gh$ of finite support complex function on the set of double cosets of $H$ in $G$ with the involution defined in (\ref{eqn:invdiscrete}) is called the {\it Hecke algebra of the discrete Hecke pair $\gh$}.
\item [(ii)] When $\gh$ is a non-discrete pair satisfying Assumptions \ref{assume:r-G-inv} and \ref{assume:H-unimod}, the involutive algebra $\H\gh$ defined by (\ref{eqn:Hgh-nondis}), (\ref{eqn:Hgh-convolution-nondis}) and (\ref{eqn:inv-nondis}) is called the {\it Hecke algebra of the pair $\gh$}.
\end{itemize}
\end{definition}

We have not defined a Hecke pair in the non-discrete case yet. The reason is that unlike discrete cases, a non-discrete pair $\gh$ satisfying Assumptions \ref{assume:r-G-inv} and \ref{assume:H-unimod}, and without any extra conditions, gives rise to a Hecke algebra. In fact, part of the condition of $H$ being a Hecke subgroup of $G$ is encoded in the definition of $\H\gh$ in (\ref{eqn:Hgh-nondis}), see also Remark \ref{rem:condition-b}(i) below. Since our intention is to study Hecke algebras in the realm of \cs-algebras, we seek those conditions which facilitate the embedding of Hecke algebras inside a \cs-algebra. Unfortunately, we are not always able to define a left regular representation on the Hilbert space $L^2(\hg)$, as it is clear in the discrete case, see Remark \ref{rem:discretehp} below. On the other hand, the requirement that every double coset has only finitely many left cosets is not necessary for a non-discrete pair $\gh$ to admit a left regular representation, see Example \ref{exa:sl2r1} and Theorem \ref{thm:Hcompact}. Therefore we propose the following definition as the generalization of Definition \ref{def:discretehp}(ii) for non-discrete pairs $\gh$:

\begin{definition}
\label{def:heckepair-allcases}
\begin{itemize}
\item [(i)] Let $\gh$ be a non-discrete pair satisfying Assumptions \ref{assume:r-G-inv} and \ref{assume:H-unimod}. It is called a {\it non-discrete Hecke pair} if the followings hold:
    \begin{itemize}
       \item [(a)] The mapping $\la:\H\gh\to B(L^2(\hg, \nu))$ defined by $\la(f)(\xi):= f\ast \xi$ for all $f\in \H\gh$ and $\xi\in L^2(H\ba G)$ is a well defined homomorphism.
       \item [(b)] For every $x\in G$, there exists some $f\in \H\gh$ such that $f(Hx)\neq 0$.
    \end{itemize}
\item [(ii)] When $\gh$ is a non-discrete Hecke pair, the homomorphism $\la$ is called the {\it left regular representation of the Hecke pair $\gh$} and the norm completion of the image of $\H\gh$ under $\la$ in the \cs-algebra $B(L^2(\hg))$ is called the {\it reduced Hecke \cs-algebra of the Hecke pair $\gh$} and is denoted by $\csr\gh$.
\end{itemize}
\end{definition}
\begin{remark}
\label{rem:condition-b}
\begin{itemize}
\item[(i)] To see the relevance of Condition (b) of Definition \ref{def:heckepair-allcases}(i), let $H$ be an arbitrary subgroup of a discrete group $G$ and define the {\it commensurator of $H$ in $G$} by  $Comm_G(H):=\{g\in G; R(g), L(g)<\infty \}$. It is known that $Comm_G(H)$ is a subgroup of $G$ and $Comm_G(H)=G$ if and only if $H$ is a Hecke subgroup of $G$. The vector space $\H\gh$ as defined in (\ref{eqn:Hgh-nondis}) coincides with the vector space of all finite support complex functions on the set $G//H$ of double cosets of $H$ in $G$ if and only if Condition (b) holds. Otherwise, the supports of functions in $\H\gh$ would be contained in $Comm_G(H) //H$. Condition (b) is especially important in the non-discrete case that we do not have any algebraic condition to impose the same restriction.
\item[(ii)] In the case that $G$ is a discrete group or in the case that $H$ is a compact open subgroup of a locally compact group $G$, Definition \ref{def:heckepair-allcases}(i) is equivalent to Definition \ref{def:discretehp}(ii). Therefore in these cases, the pair $\gh$ can be considered both as discrete and non-discrete. Ambiguity arise only when $H$ is an open Hecke subgroup of $G$. Because, in this case the topology and measure structure of $G$ and $H$ are ignored (both are considered discrete), and so we do not have to worry about Assumptions \ref{assume:r-G-inv} and \ref{assume:H-unimod}. That is why we have separated the definition of discrete and non-discrete Hecke pairs.
\item[(iii)] When $H$ is a compact or cocompact subgroup of a locally compact group $G$, Condition (b) of Definition \ref{def:heckepair-allcases}(i) readily holds for the pair $\gh$.
\end{itemize}
\end{remark}

Assume $\gh$ is a non-discrete Hecke pair satisfying Condition (b) of Definition \ref{def:heckepair-allcases}(i). To check Condition (a), we must first prove that $\la(f)$ is a bounded operator on the Hilbert space $L^2(\hg, \nu)$ for all $f\in \H\gh$. Next, a typical argument based on Fubini's theorem, shows that $\la$ is a homomorphism.
\begin{remark}
\label{rem:star-hom}
    Given a discrete (resp. non-discrete) Hecke pair $\gh$, a straightforward computation shows that the left regular representation $\la$ is a $\ast$-homomorphism whenever the discrete Hecke pair $\gh$ is relatively unimodular (resp. the group $G$ is unimodular).
\end{remark}
Regarding the similarity between Hecke pairs and  groups, one wishes to prove that $\la$ is a bounded operator, that is to find a constant $M>0$ such that $\la(f)\leq M\|f\|_1$. Although we are not able to prove this last statement for general Hecke pairs, we shall prove it in some important cases, see Theorem \ref{thm:Hcompact}, Theorem \ref{thm:lambdabounded-dis} and the proof of Theorem \ref{thm:cocompact}. One also notes that Assumption \ref{assume:H-unimod} is needed only if we are interested to have an involution on the Hecke algebra $\H\gh$.
\begin{remark}
\label{rem:discretehp}
Assume $G$ is a discrete group, as it has already been studied in the literature, the pair $\gh$ is a Hecke pair in the sense of Definition \ref{def:heckepair-allcases}(i) if and only if it is a Hecke pair in the sense of Definition \ref{def:discretehp}(ii). The ``if'' part was proved in Proposition 1.3.3 of \cite{curtis}. To see the ``only if'' part, assume that there is a double coset $HxH$ which contains infinitely many distinct right cosets, say $\{Hx_k\}_{k=1}^\infty$. For every $k\in \n$, let $\chi_k\in \H\gh$ be the characteristic function of the double coset $Hx_kH$ and let $\d_e\in \ell^2(\hg)$ be the characteristic function of the right coset $H$. Then one easily computes $\chi_k\ast \d_e(Hx_k)=1$ for all $k\in \n$, and so $\|\chi_k\ast \d_e\|_2=\infty$.
\end{remark}
\subsection*{Reduction of Hecke pairs}
\label{subsec:reduction-Hecke-pairs}
The following proposition is a useful tool to realize more pairs $\gh$ of locally compact groups as Hecke pairs. It is also applied in the reduction of Hecke pairs.
\begin{proposition}
\label{prop:normalcomponent}
Let $\gh$ be a pair and let $N$ be a normal closed subgroup of $G$ contained in $H$. Set $G':=\frac{G}{N}$ and $H':=\frac{H}{N}$.
\begin{itemize}
\item [(i)] If the pair $\gh$ satisfies Assumptions \ref{assume:r-G-inv} and \ref{assume:H-unimod}, then the pair $(G',H')$ satisfies them too.
\item [(ii)] If $N$ is unimodular and $H'$ is compact, then both pairs $(G',H')$ and $\gh$ satisfy Assumptions \ref{assume:r-G-inv} and \ref{assume:H-unimod}.
\item [(iii)] Assume that the pair $\gh$ satisfies Assumptions \ref{assume:r-G-inv} and \ref{assume:H-unimod}, then the pair $\gh$ is a Hecke pair if and only if the pair $(G',H')$ is a Hecke pair. In this case, the Hecke algebras $\H\gh$ and $\H(G',H')$ are canonically isomorphic.
\item[(iv)] With the assumptions of Item (iii), the left regular representation of $\H\gh$ is bounded if and only if the left regular representation of $\H(G',H')$ is bounded. Moreover, the \cs-algebras $\csr\gh$ and $\csr(G',H')$ are canonically isomorphic.
\end{itemize}
\end{proposition}
\begin{proof}
Only Items (iii) and (iv) are applied to discrete pairs and their proof in this case is easy. Therefore, we only prove the proposition for non-discrete pairs.
\begin{itemize}
\item [(i)] Let $\gh$ satisfies Assumptions  \ref{assume:r-G-inv} and \ref{assume:H-unimod}. Since $H$ is unimodular and $N$ is normal in $H$, both $N$ and $H'$ are unimodular too, and so the assumption \ref{assume:H-unimod} holds for $(G',H')$.

    To check Assumption \ref{assume:r-G-inv}, it is enough to show that $H'\ba G'$ possesses a right $G'$-invariant measure. In the following argument $x$ is always an arbitrary element of $G$. Let $\pi:G\to G'$, $x\mapsto \bar{x}$ be the quotient map. Define $\pi':\hg \to H'\ba G'$ by $Hx\mapsto H' \bar{x}$. It is easy to see that $\pi'$ is a well defined bijection. Consider the commutative diagram
    \begin{equation}
    \label{diag:piprime}
    \xymatrix{
    G \ar[rr]^{\pi} \ar[d]&& G'\ar[d]  \\
    H\ba G \ar[rr]^{\pi'} && H'\ba G'
    }
    \end{equation}
    Since $\pi$ and the vertical arrows are continuous open mappings, $\pi'$ is a homeomorphism. Therefore it gives rise to a linear isomorphism
    \[
    \ff:C_c(\hg)\to C_c(H'\ba G'),\qquad \ff(f)(H'\bar{x}):=f(Hx),
    \]
    for all $f\in C_c(\hg)$ and $H'\bar{x}\in H'\ba G'$. In the following $g$ represents an arbitrary element of $C_c(H'\ba G')$. First we note that the inverse of $\ff$ is given by $\ff\inv(g)=g\circ \pi'$. Next, we define a functional $I: C_c(H'\ba G')\to \c$ by
    \[
    I(g):=\int_\hg g\circ \pi'(Hx) d\nu(x).
    \]
    One checks that $I$ is a positive linear functional on $C_c(H'\ba G')$. Furthermore, $I(g)=0$ if and only if $g=0$ almost everywhere. It is straightforward to check that $I$ is invariant by the action of $G'$ on $C_c(H'\ba G')$ induced by the right action of $G'$ on $H'\ba G'$, that is $I(R_{\bar{y}}(g)) = I(g)$ for all $\bar{y}\in G'$, here as usual $R_{\bar{y}} (g) (H\bar{x}) = g(H\bar{x}\bar{y})$ for all $H'\bar{x}\in H'\ba G'$. Therefore the positive functional $I$ defines a right $G'$-invariant Radon measure $\th$ on $H'\ba G'$, which amounts to saying that $\D_{G'}|_{H'}=\D_{H'}$.
\item [(ii)] Assumptions \ref{assume:r-G-inv} and \ref{assume:H-unimod} for $(G',H')$ easily follow from Proposition 2.27 of \cite{folland-ha}. Thus we only need to prove these assumptions for the pair $\gh$. Since $N$ is unimodular, $N\subseteq Ker\D_G$. Therefore there is a continuous group homomorphism $\rho:G'\to \r^+$ such that $\rho \circ \pi= \D_G$. Since $H'$ is a compact subgroup of $G'$, $\rho(H')$ is a compact subgroup of $\r^+$. But the only compact subgroup of $\r^+$ is the trivial subgroup $\{1\}$. Hence $\rho|_{H'}=1$ and this implies that $\D_G|_H=1$. Similarly one shows that $\D_{H}=1$.
\item [(iii)] Regarding the way we defined the measure $\th$ on $H'\ba G'$ in the above, we have
    \[
    \int_{H'\ba G'} \ff(f)(H'\bar{x}) d\th(H'\bar{x})= \int_\hg f(Hx) d\nu(Hx),
    \]
    for all $f\in C_c(\hg)$. Hence the isomorphism $\ff$ extends to an isometric isomorphism $\ff_2$ between $L^2(\hg)$ and $L^2(H'\ba G')$. On the other hand, one easily checks that $\ff$ maps $\H\gh$ onto $\H(G', H')$. Let  $\ff_\H:\H\gh\to \H(G',H')$ denote the linear isomorphism obtained by restricting $\ff$ to $\H\gh$. Then a straightforward computation shows that $\ff_\H$ preserves the convolution product, so it is an algebra isomorphism and we have
    \[
    \ff_\H(f)\ast \ff_2(\xi)= \ff_2(f\ast \xi), \qquad \forall f\in \H\gh, \xi\in L^2(\hg).
    \]
    The above discussion shows that Condition (a) of Definition \ref{def:heckepair-allcases}(i) holds for $\gh$ if and only if it holds for $(G',H')$. The equivalence of Condition (b) of Definition \ref{def:heckepair-allcases}(i) for pairs $\gh$ and $(G',H')$ follows from the fact that the mapping $\ff_{\H}$ is an isomorphism coming from the canonical surjection $\pi$.
\item[(iv)] Similar to the above item, $\ff$ extends to an isometric isomorphism between $L^1(\hg)$ and $L^1(H'\ba G')$. Then the desired statements follow immediately from (iii).
\end{itemize}
\end{proof}

\begin{remark} With the notations of the above proposition, we note the followings:
\begin{itemize}
\item [(i)] To see that the compactness of $H'$ is necessary in Proposition \ref{prop:normalcomponent}(ii), consider the group $H=\r \rtimes\r^\times$ and its normal subgroup $N=\r$. The locally compact group $H$ is not unimodular. However, both $N$ and $\frac{H}{N}=\r^\times$ are unimodular.
\item [(ii)] It is straightforward to see that the map $\ff_\H$ preserves involution if and only if $\D_G(x)=\D_{G'}(\bar{x})$ for all $x\in G$. However, the above example also shows that this formula is not generally true.
\end{itemize}
\end{remark}

There are also other ways of constructing new Hecke pairs from given ones. For example, direct product of finitely many Hecke pairs or considering extensions of discrete groups by discrete Hecke pairs, as it was done for discrete Hecke pairs in Section 3 of \cite{s2}.


\section{Unimodularity and relative unimodularity}
\label{sec:relativemod}

In this section, $\gh$ is always a discrete Hecke pair, and for $g\in G$, we denote the characteristic function of the double coset $HgH$ by $\chi_g$. To motivate the subject of this section, besides Remark \ref{rem:star-hom}, we note the following points: Two involutions $\st$ and $\s$ defined on $\H\gh$ in (\ref{eqn:traditionalinv}) and (\ref{eqn:invdiscrete}), respectively, agree if and only if the Hecke pair $\gh$ is relatively unimodular. On the other hand, it is easy to check that this latter condition holds if and only if the involution $\st$ preserves the $\ell^1$-norm of $\hg$, see also Example \ref{exa:involutionnot}. Moreover, when $H$ is a compact open subgroup of a locally compact group $G$, the equality $\D_\gh=1$ is equivalent to the unimodularity of $G$. Furthermore, the relative unimodularity is a necessary condition for property (T) and property (RD) of discrete Hecke pairs, See Proposition \ref{prop:PTunimod} below and our results in \cite{s4}. Therefore we study various criteria implying this condition. We shall also briefly address the non-discrete pairs $\gh$ for which $G$ is unimodular. The following definition introduces an algebraic notion which is equivalent to the relative unimodularity of discrete Hecke pairs:
\begin{definition}
\label{def:semicom} The Hecke pair $\gh$ as well as its Hecke algebra, $\H\gh$, and its reduced Hecke \cs-algebra, $\csr\gh$, are called {\it locally commutative} if $\chi_{g} \ast \chi_{g\inv}(H)=\chi_{g\inv} \ast \chi_{g}(H)$ for all $g\in G$.
\end{definition}

\begin{proposition}
\label{prop:communimod}
    A discrete Hecke pair $\gh$ is relatively unimodular if and only if it is locally commutative.
\end{proposition}
\begin{proof} For a given $g\in G$, assume that the double coset $HgH$ is the disjoint union of right cosets $Hx_i$ for $i=1,\cdots, R(g)$. We compute
\begin{eqnarray*}
\chi_{g\inv} \ast \chi_{g} (H)&=&\sum_{Hy\in \hg}  \chi_{g\inv} (Hy\inv H) \chi_{g} (Hy) \\
&=& \sum_{i=1}^{R(g)} \chi_{g\inv} (Hx_i\inv H) = R(g).
\end{eqnarray*}
By replacing $g$ with $g\inv$, we get $L(g)=\chi_{g} \ast \chi_{g\inv} (H)$. The desired statement follows from these equalities.
\end{proof}
Clearly every commutative Hecke algebra is locally commutative. Therefore the above proposition applies to the following example:
\begin{example}
\label{exa:Heckeoriginal}
Let $G=GL_2(\q)^+:=\{g\in GL_2(\q); \det(g)>0\}$ and $H=SL_2(\z)$. Then $\gh$ is a Hecke pair and its Hecke algebra $\H\gh$ is commutative, see Proposition 1.4.1 and Theorem 1.4.2 of \cite{bump}. In fact, the terminology ``Hecke subgroup'', ``Hecke pair'' and ``Hecke algebra'' has been originated from this Hecke pair which was introduced for the first time in the realm of modular forms by E. Hecke, for more historical notes see \cite{krieg}.
\end{example}
Besides above examples, commutative Hecke algebras appear in numerous situations, for instance see Proposition 3.5 of \cite{tzanev}.
\begin{definition}
\label{def:Gelfandpair}
A (discrete or non-discrete) Hecke pair $\gh$ is called a {\it Gelfand pair}, if the Hecke algebra $\H\gh$ is commutative.
\end{definition}

Now, we propose two operator algebraic conditions implying relative unimodularity of Hecke pairs:
\begin{corollary}
\label{cor:normal-unimod}
    \begin{itemize}
        \item[(i)] If for every $g\in G$, the element $\chi_g\in\H\gh$ is normal, i.e. $\chi_g\ast\chi_g^\ast=\chi_g^\ast\ast\chi_g$, then $\gh$ is relatively unimodular and the image of $\chi_g$ in $\csr\gh$ (under left regular representation) is a normal operator.
        \item[(ii)] For every $g\in G$, $\chi_g$ is a self adjoint element in $\H\gh$, i.e. $\chi_g^\ast=\chi_g$ if and only if
        \begin{equation}
        \label{eqn:doubleinv}
        HgH=Hg\inv H, \quad \forall g\in G.
        \end{equation}
        \item[(iii)] If $\gh$ satisfies the equivalent conditions mentioned in (ii), then $\gh$ is a Gelfand pair, relatively unimodular and the image of $\chi_g$ in $\csr\gh$ is a self adjoint operator for all $g\in G$.
    \end{itemize}
\end{corollary}
\begin{proof}
    \begin{itemize}
        \item[(i)] This follows from Proposition \ref{prop:communimod} and Identity \ref{eqn:invcharacteristic}. Then by Remark \ref{rem:star-hom},  the image of $\chi_g$ is a normal operator in $\csr\gh$ for all $g\in G$.
        \item[(ii)] Assume all $\chi_g$ are self adjoint, so they are normal and by (i) we get $\D_\gh=1$. Then by Identity \ref{eqn:invcharacteristic}, we have $\chi_g=\chi_{g\inv}$ for all $g\in G$ which is equivalent to Equation (\ref{eqn:doubleinv}).

            Equation (\ref{eqn:doubleinv}) clearly implies that $\gh$ is locally commutative and so relatively unimodular. By applying Identity \ref{eqn:invcharacteristic}, we see that $\chi_g$ is a self adjoint element for all $g\in G$.
        \item[(iii)] To show that the Hecke algebra $\H\gh$ is commutative, we only need to show $\chi_g\ast\chi_k=\chi_k\ast\chi_g$ for all $g, k \in G$. First, we note that $\chi_g\ast\chi_k$ is a real valued function. Thus $\chi_g\ast\chi_k=\sum_{HxH\in G//H} c_x \chi_x$, where all coefficients $c_x$ are real numbers and only finitely many of them are non-zero. This shows that $\chi_g\ast\chi_k$ is a self adjoint element, that is $(\chi_g\ast\chi_k)^\ast=\chi_g \ast\chi_k$. On the other hand, we have $(\chi_g\ast\chi_k)^\ast=\chi_k^\ast \ast\chi_g^\ast=\chi_k \ast\chi_g$. This proves that $\gh$ is a Gelfand pair. The rest of the statement is easy.
    \end{itemize}
\end{proof}
We borrowed the proof of Item (iii) from \cite{hall}. Although the involution defined in \cite{hall} is the involution $\st$, defined in Equation (\ref{eqn:traditionalinv}).
\begin{example}
\label{exa:scommut}
\begin{itemize}
\item[(i)] Consider the Hecke pair $(SL_2(\mathbb{Q}_p), SL_2(\mathbb{Z}_p))$, where $p$ is a prime number. The group $SL_2(\mathbb{Q}_p)$ is a totally disconnected locally compact group and $SL_2(\mathbb{Z}_p)$ is a compact open subgroup of this group. Condition (\ref{eqn:doubleinv}) was verified for this Hecke pair in Proposition 2.10 of \cite{hall}.
\item[(ii)] Let $G$ be an abelian group.  Consider the action $\th$ of $\z/2$ on $G$ by inversion, that is $\th_{-1}(g):=g\inv$ for all $g\in G$, where $\z/2$ is considered as the multiplicative group $\{1,-1\}$. One easily checks that Condition (\ref{eqn:doubleinv}) holds for every Hecke pair of the form $(G\rtimes \z/2, \{0 \}\rtimes \z/2)$.
\end{itemize}
\end{example}
The following example shows that Condition (\ref{eqn:doubleinv}) is not necessary for commutativity of the Hecke algebras $\H\gh$:
\begin{example}
\label{exa:commhp1}
Consider the semi-crossed product group $G=(\z/5\times \z/5) \rtimes \z/2$ defined by the action of $\z/2$ on $\z/5\times \z/5$ by flipping the components, that is $(-1)(x,y):=(y,x)$ for all $x,y\in \z/5$, where again the acting copy of $\z/2$ is assumed to be the multiplicative group $\{1,-1\}$. The Hecke algebra associated to the Hecke pair $((\z/5\times \z/5) \rtimes \z/2, \{(0,0)\}\rtimes \z/2)$ is commutative, see for instance Proposition 3.5 of \cite{tzanev}. However one easily checks that Condition (\ref{eqn:doubleinv}) does not hold for this Hecke pair.
\end{example}

An immediate corollary of Equality (\ref{eqn:relativeequality}) and Proposition \ref{prop:communimod} is:
\begin{corollary}
\label{cor:LCG-unimod1}
Let $G$ be a locally compact group which possesses a compact open subgroup $H$. If the discrete Hecke pair $\gh$ is locally commutative, then $G$ is unimodular. In particular, if a discrete Hecke pair $\gh$ is a Gelfand pair, then $G$ is unimodular.
\end{corollary}
One notes that the above corollary mostly applies to totally disconnected locally compact groups, because they have plenty of compact open subgroups.
\begin{remark}
\label{rem:quotient-semi-comm}
When $H$ is a normal compact open subgroup of a locally compact group $G$, the discrete Hecke pair $\gh$ is always locally commutative, and therefore $G$ must be unimodular. On the other hand, in this case $\H\gh$ is isomorphic to the complex group algebra of the quotient group $G/H$, so it can be a noncommutative algebra. This is why locally commutative Hecke pairs is more general than Gelfand pairs. In fact, by Proposition \ref{prop:communimod}, all discrete Hecke pairs described in Proposition \ref{prop:alg-relative-unimod} below are locally commutative.
\end{remark}
In the following we describe two more algebraic conditions implying the relative unimodularity. They come from the work of G.M. Bergman and H.W. Lenstra in \cite{berlenstra}. First we need some definitions:

\begin{definition}
\label{def:commen-nearly}
\begin{itemize}
\item [(i)] Two subgroups $H$ and $K$ of a group $G$ are called {\it commensurable} if $H\cap K$ is a finite index subgroup of both $H$ and $K$.
\item [(ii)] A subgroup $H$ of a group $G$ is called {\it nearly normal} if it is commensurable with a normal subgroup $N$ of $G$.
\item [(iii)] A subgroup $H$ of a group is called {\it almost normal} if it has only finitely many conjugates.
\end{itemize}
\end{definition}
Condition (iii) in the above definition should not be confused with the definition of almost normal subgroups according to \cite{bc}, see Definition \ref{def:discretehp}(ii) and its footnote. One also observes that $H$ is an almost normal subgroup of $G$ if and only if the normalizer of $H$ in $G$, is a finite index subgroup of $G$.
\begin{proposition}
\label{prop:alg-relative-unimod}
A discrete Hecke pair $\gh$ is relatively unimodular if one of the following conditions holds:
\begin{itemize}
\item [(i)] The subgroup $H$ is nearly normal in $G$.
\item [(ii)] The subgroup $H$ is almost normal in $G$.
\end{itemize}
In particular, if a locally compact group $G$ possesses an open Hecke subgroup $H$ which is also nearly normal or almost normal in $G$, then $G$ is unimodular.
\end{proposition}

\begin{proof} Since the notions defined in Definition \ref{def:commen-nearly} are algebraic, we do not have to worry about the topology in the following arguments:
\begin{itemize}
\item [(i)] According to Theorem 3 of \cite{berlenstra}, $H$ is nearly normal if and only if there exists a natural number $n$ such that $1\leq L(g)\leq n$ for all $g\in G$. This statement also follows from the work of G. Schlichting in \cite{sch2}. Thus $\D_\gh(g)\leq n$ for all $g\in G$. On the other hand, since $\D_\gh$ is a group homomorphism from $G$ into the multiplicative group $\q^+$, the boundedness of its image amounts to $\D_\gh=1$.
\item [(ii)] It follows from (i) and the fact that if $H$ is an almost normal and a Hecke subgroup of $G$, then it is nearly normal.
\end{itemize}
\end{proof}

In \cite{neu}, B. H. Neumann determined the class of all finitely generated groups all whose subgroups are nearly normal, see also Theorem 2.8 of \cite{s2} for a summary. Also, in Example 2.10 of \cite{s2}, we explained a method to construct nearly normal subgroups of free product of two discrete groups. In its simplest form, it is based on the fact that every finite subgroup of a finitely generated group gives rise to a nearly normal subgroup of a finitely generated free group.

The above discussion suggests a similar study of the unimodularity of locally compact groups. Regarding Assumption \ref{assume:H-unimod}, for a given non-discrete Hecke pair $\gh$, we are interested in cases where both $G$ and $H$ are unimodular.
\begin{remark}
\label{rem:unimod} Besides Corollary \ref{cor:LCG-unimod1} and elementary cases where $H$ is either a discrete, normal, or compact subgroup of a unimodular group $G$, see Corollary 1.5.5 of \cite{dit-ech} and Proposition 2.27 and Corollary 2.28 of \cite{folland-ha}, we have the simultaneous unimodularity of $G$ and $H$ in the following cases:
\begin{itemize}
\item[(i)] When $H$ is a lattice in $G$. A discrete subgroup $H$ of a locally compact group $G$ is called a {\it lattice in $G$} if the homogeneous space $\hg$ has a $G$-invariant Radon measure $\nu$ such that $\nu(\hg)<\infty$. In this case, by Theorem 9.1.6 of \cite{dit-ech}, $G$ is unimodular too.
\item [(ii)] When $H$ is a unimodular and cocompact subgroup of a locally compact group $G$. In this case, by Proposition 9.1.2 of \cite{dit-ech}, $G$ is unimodular too.
\item [(iii)] When $G$ is unimodular and $H$ is a nearly normal closed subgroup of $G$. Let $N$ be a normal subgroup of $G$ commensurable with $H$, which is also closed. Then the statement follows from Lemma \ref{lem:unimodcomm} below.
\end{itemize}
\end{remark}

\begin{lemma}
\label{lem:unimodcomm}
Let $H$ and $K$ be two commensurable closed subgroups of a locally compact group $G$. Then $H$ is unimodular if and only if $K$ is unimodular.
\end{lemma}
\begin{proof}
Without loss of generality, we can assume that $K$ is a finite index subgroup of $H$. By replacing $K$ with $N:=\bigcap_{h\in H} hKh\inv $, we can even assume $K$ is normal. Then the statement follows from Proposition 9.1.2 and Corollary 1.5.5 of \cite{dit-ech}.
\end{proof}

\section{When $H$ is a compact subgroup of $G$}
\label{sec:Hcompact}

In this section we show that when $H$ is a compact subgroup of a locally compact group $G$, the pair $\gh$ is a Hecke pair. The special case where $H$ is a compact and open subgroup of $G$ will be treated with more details in the next section, because it is better studied within the class of Hecke pairs with open Hecke subgroups.

Besides compactness of $H$, we also assume that the triple $(\eta, \mu, \nu)$ of measures satisfies Weil's formula, Equation (\ref{eqn:invmeasure}). We define a map
\begin{equation}
\label{eqn:iota}
\iota:C_c(\hg ) \to C_c(G),
\end{equation}
\[
\iota(f)(x)= \tilde{f}(x) :=f(Hx), \qquad \forall f\in C_c(\hg), \,x\in G.
\]
Since $H$ is compact, this map is well defined. We also consider its extensions from $L^2(\hg )$ into $L^2(G)$ and from $L^1(\hg )$ into $L^1(G)$, and denote them with the same notation. Since $\tilde{f}$ is left $H$-invariant, one checks that
\[
\int_H |\tilde{f} (hx)|^2 d\eta(h)= |\tilde{f} (x)|^2\int_H d\eta(h)= \eta(H) |f(Hx)|^2, \quad \forall x\in G.
\]
Thus it follows from Weil's formula that $\|\tilde{f}\|_2^2= \eta(H) \|f\|_2^2$ for all $f\in L^2(\hg)$ where the $L^2$-norms of $\tilde{f}$ and $f$ are computed in $L^2(G)$ and $L^2(\hg)$, respectively. Similarly, we have $\|\tilde{f}\|_1=\eta(H)\|f\|_1$. One also notes that for every $f\in \H\gh$, $\tilde{f}$ is a bi-$H$-invariant function on $G$, and so $f(Hx)=\tilde{f}(xh)$ for all $x\in G$ and $h\in H$. Therefore for every $f\in \H\gh$, $g\in L^2(\hg)$ and $x\in G$, one computes
\begin{eqnarray*}
\widetilde{f\ast g}(x)&=&f\ast g(Hx)= \int_\hg f(Hxy\inv) g(Hy) d\nu(y)\\
&=&\frac{1}{\eta(H)} \int_\hg \tilde{f}(xy\inv) \tilde{g}(y)\left( \int_H d\eta(h) \right) d\nu(y)\\
&=& \frac{1}{\eta(H)}\int_\hg \left( \int_H \tilde{f}(xy\inv h\inv) \tilde{g}(hy) d\eta(h)\right)  d\nu(y)\\
&=& \frac{1}{\eta(H)}\int_G \tilde{f}(xy\inv) \tilde{g}(y) d\mu(y)\\
&=& \frac{1}{\eta(H)} \tilde{f}\ast \tilde{g}(x).
\end{eqnarray*}
Hence
\begin{eqnarray*}
\|f \ast g \|_2^2&=& \frac{1}{\eta(H)} \| \widetilde{f\ast g} \|_2^2=  \frac{1}{\eta(H)^3} \| \tilde{f}\ast \tilde{g} \|_2^2\\
&\leq& \frac{1}{\eta(H)^3} \|\tilde{f}\|^2_1 \|\tilde{g}\|_2^2 = \|f\|^2_1 \|g\|_2^2.
\end{eqnarray*}
We summarize the above computations in the following theorem:
\begin{theorem}
\label{thm:Hcompact} Let $H$ be a compact subgroup of a locally compact group $G$. Then the pair $\gh$ is a Hecke pair. Furthermore, the left regular representation $\la:\H\gh \to B(L^2(\hg))$ is bounded and we have $\|\la(f)\|\leq \|f\|_1$.
\end{theorem}
One notes that the above theorem is valid particularly when $H$ is a compact open subgroup of $G$, whether the Hecke pair $\gh$ is considered as a discrete or non-discrete Hecke pair. In fact, this is where two settings of locally compact Hecke pairs and discrete Hecke pairs coincide with each other, and it allows us to transfer many results and ideas from one setting to another.
\begin{remark}
\label{rem:HcompactL1}
When $\gh$ is as the above theorem, we define $L^1(G//H)$ as the involutive Banach algebra consisting of those functions in $L^1(\hg)$ that are almost everywhere right $H$-invariant. Then by the continuity of $\la$, we can extend the left regular representation to whole $L^1(G//H)$.
\end{remark}

Examples of pairs $\gh$ of locally compact groups subject to the above theorem have already studied in other branches of mathematics such as noncommutative harmonic analysis, representations theory, Lie theory, geometric group theory and random walks. Therefore not only we have no shortage of examples for these Hecke pairs, but we also have plenty of opportunities to find new notable applications for our extended \cs-algebraic formulation of Hecke algebras. There is an important class of Hecke pairs of the type described in the above theorem. This class consists of the Schlichting completions of reduced discrete Hecke pairs $\gh$, which will be addressed in the following section.

\section{When $H$ is an open subgroup of $G$}
\label{sec:Hopen}

In this section $\gh$ is always a discrete Hecke pair. Using the Schlichting completion of $\gh$, we reduce the study of this type of Hecke pairs to the case that the subgroup in the Hecke pair is compact and open. Then we apply the results obtained in previous sections to discrete Hecke pairs.

\begin{definition}
\label{def:reducedHP}
Given a discrete Hecke pair $\gh$, the normal subgroup of $G$ defined by $K_\gh:=\bigcap_{x\in G} xHx\inv $ is called the {\it core of the Hecke pair $\gh$}. The Hecke pair $\gh$ is called {\it reduced} if its core is the trivial subgroup. The pair $(\frac{G}{K_\gh},\frac{H}{K_\gh})$ is a reduced discrete Hecke pair which is called the {\it reduction of $\gh$} or the {\it reduced Hecke pair associated with $\gh$} and is denoted by $\ghr$.
\end{definition}

The reduced Hecke pair $\ghr$ associated with $\gh$ has a similar properties as $\gh$ in a number of situations, for instance see Sections 3 and 4 of \cite{s4}. In fact, it follows from Proposition \ref{prop:normalcomponent} that the Hecke algebras of Hecke pairs $\gh$ and $\ghr$ are naturally isomorphic. Therefore we often assume that discrete Hecke pairs are reduced.

The Schlichting completion is a process to associate a totally disconnected locally compact group $\gb$ with a given reduced discrete Hecke pair $\gh$ such that $G$ is dense in $\gb$, and more importantly, the closure of $H$ in $\gb$, denoted by $\hb$, is a compact open subgroup of $\gb$. Then the discrete Hecke pair $\ghb$ is called the {\it Schlichting completion of $\gh$}. We refer the reader to \cite{klq} for a detailed account of the Schlichting completion. The following examples explains some aspects of the subject.
\begin{example}
\label{exa:Schlichting} Let $\gh$ be a discrete Hecke pair.
\begin{itemize}
\item [(i)] Assume that $G$ is a discrete group and $\gh$ is a reduced Hecke pair. Then it follows that the Schlichting completion of $\gh$ is the same as $\gh$ if and only if $H$ is a finite group.
\item [(ii)] Let $H$ be a nearly normal subgroup of $G$. If $H$ is finitely generated, then the subgroup $H_r$ in the reduced Hecke pair $\ghr$ associated to $\gh$ is finite, and consequently we have $(\gb_r, \hb_r)=\ghr$. To prove this, it is enough to show that $H$ has a finite index subgroup $L$ which is normal in $G$. Let $N$ be a normal subgroup of $G$ which is commensurable with $H$. Since $H$ is finitely generated, $N$ is finitely generated too, and therefore there are only finitely many subgroups of index equal to $[N:H\cap N]$ in $N$. Set
    \[
    L:=\cap_{x\in G} x(H\cap N)x\inv.
    \]
    Then $L$ is finite index in $H$ and normal in $G$. This proves the above claim.

    When $H$ is not finitely generated, the Hecke pair $\gh$ can be reduced and its Schlichting completion is different than itself. For example, let $N$ be an infinite product of $\z/2$ and let $A$ be the full automorphism group of $N$. Set $G:=N\rtimes A$ and let $H$ be a subgroup of $N$ of index 2. Then since $A$ acts transitively on nontrivial elements of $N$, $H$ does not contain any non-trivial normal subgroup of $G$.\footnote{We thank Henry Wilton and Jeremy Rickard for helpful conversations about these points in mathoverflow.net.}
    
    One notes that when $H$ is a finite index subgroup of $G$, the condition of $H$ being finitely generated is unnecessary. Because, in any case, the subgroup $K_\gh$ is a finite index subgroup of $H$. Thus one gets $(\gb_r, \hb_r)=\ghr$.
\item [(iii)] Let $\gh$ be a the Bost-Connes Hecke pair discussed in Example \ref{exa:involutionnot}. One checks that it is a reduced Hecke pair and it is shown in Example 11.4 of \cite{klq} that its Schlichting completion $\ghb$ is as follows:
    \[
    \gb=\left\{ \left( \begin{array} {rr}1&b\\0&a \end{array}\right); a\in \q^+, b\in \mathcal{A} \right\}, \quad \text{and} \quad \hb=\left\{ \left( \begin{array} {rr}1& r\\0&1
\end{array}\right); r\in \mathcal{Z} \right\},
    \]
    where $\mathcal{A}$ and $\mathcal{Z}$ are the ring of finite adeles on $\q$ and its the maximal compact subring, respectively. Moreover, the topology of the above Schlichting completion coincides with the topology coming from the topology of $\mathcal{A}$, i.e. the restricted product topology of $p$-adic fields.
\item [(iv)] Given a prime number $p$, consider the Hecke pair $(SL_2(\z [1/p]), SL_2(\z))$. By looking at intersections of the form $SL_2(\z) \bigcap x_n SL_2(\z)x_n\inv$, where $x_n=\left( \begin{array} {rr}0&p^{-n}\\-p^n&0 \end{array}\right)$ for $n\in \n$, one easily observes that $K_{(SL_2(\z [1/p]), SL_2(\z))}=\{I, -I\}$, where $I$ is the $2\times 2$ identity matrix. Therefore the pair $(PSL_2(\z [1/p]), PSL_2(\z))$ is the reduction of the Hecke pair $(SL_2(\z [1/p]), SL_2(\z))$. Then it is shown in Example 11.8 of \cite{klq} that the Schlitching completion of $(PSL_2(\z [1/p]), PSL_2(\z))$ is $(PSL_2(\q_p), PSL_2(\z_p))$, which is the reduction of the Hecke pair appearing in Example \ref{exa:scommut}(i). Similar discussion shows that $(PSL_2(\mathcal{A}), PSL_2(\mathcal{Z}))$ is the Schlichting completion of $(PSL_2(\q), PSL_2(\z))$. One notes that the Hecke topology of these Schlichting completions again coincide with the conventional totally disconnected locally compact topologies of the above groups coming from $p$-adic valuations and the restricted product of $p$-adic fields.
\end{itemize}
\end{example}

The following lemma simply states that passing to the Schlichting completion does not change the algebraic and analytic aspects of reduced discrete Hecke pairs and their associated Hecke algebras.
\begin{lemma}
\label{lem:bijective-Schlich} (\cite{klq}, Proposition 4.9) Let $\gh$ be a reduced discrete Hecke pair and let $\ghb$ be its Schlichting completion. Then the following statements hold:
\begin{itemize}
\item[(i)] The mapping $\alpha: \hg\to \hb\ba\gb$ (resp. $\alpha':G/H\to \gb / \hb$), defined by $Hg\mapsto \hb g$ (resp. $gH\mapsto g\hb$) for all $g\in G$ is a $G$-equivariant bijection. In particular, $\alpha$ induces an isometric isomorphism between Hilbert spaces $\ell^2(\hg)$ and $\ell^2(\hb \ba \gb)$.
\item[(ii)] The mapping $\beta: G//H \to \gb// \hb$, defined by $HgH\mapsto \hb g\hb$ for all $g\in G$ is a bijection.
\item[(iii)] The mapping $\beta$ commutes with the convolution product, and therefore it induces an isometric isomorphism between the Hecke algebras $\H\gh$ and $\H\ghb$, with respect to the corresponding $\ell^1$-norms.
\end{itemize}
\end{lemma}
The following theorem is an important application of the Schlichting completion and Theorem \ref{thm:Hcompact}:

\begin{theorem}
\label{thm:lambdabounded-dis} Let $\gh$ be a discrete Hecke pair. Then the left regular representation $\la:\H\gh\to B(\ell^2(\hg))$ is bounded. In fact, we have $\|\la(f)\|\leq \|f\|_1$ for all $f\in \H\gh$.
\end{theorem}
\begin{proof}
By Proposition \ref{prop:normalcomponent}, without loss of generality, we can assume $\gh$ is a reduced discrete Hecke pair. Let $\alpha$ and $\beta$ be as Lemma \ref{lem:bijective-Schlich}. Since both Hecke pairs $\gh$ and $\ghb$ are discrete, the above mappings induce the following isometric isomorphisms:
\begin{eqnarray*}
\alpha\s:\ell^2(\hb \ba \gb) &\to& \ell^2(\hg),  \qquad \alpha\s(\xi):= \xi \circ \alpha, \quad \forall \xi \in \ell^2(\hb \ba \gb), \\
\beta\s:\H\ghb &\to& \H\gh, \qquad  \beta\s(f):=f\circ \beta, \quad \forall f\in \H\ghb,
\end{eqnarray*}
where the norm on Hecke algebras $\H\gh$ and $\H\ghb$ are the corresponding $\ell^1$-norms. One notes that a similar isometry as $\alpha\s$ exists from $\ell^1(\hb\ba \gb)$ onto $\ell^1(\hg)$. It is straightforward to check that this isometries commute with the convolution products defining the left regular representations of $\H\gh$ and $\H\ghb$. Since $\hb$ is compact, by applying Theorem \ref{thm:Hcompact} and the above discussion for every $f\in \H\gh$ and $\xi\in\ell^2(\hg)$, we have
\begin{eqnarray*}
\|f\ast \xi\|_2&=&\|\alpha\s(({\beta\s}\inv f)\ast({\alpha\s}\inv\xi))\|_2\\
&=&\|({\beta\s}\inv f)\ast({\alpha\s}\inv\xi)\|_2\\
&\leq &\|{\beta\s}\inv f\|_1 \|{\alpha\s}\inv\xi\|_2\\
&= &\|f\|_1 \|\xi\|_2.
\end{eqnarray*}
\end{proof}
It worths mentioning that the above theorem is stronger than Proposition 1.3.3 of \cite{curtis}, as it compares the operator norm of $\lambda(f)$ with $\|f\|_1$. This is important because it generalizes the same inequality in locally compact groups. Moreover, this result is a necessary step to show that Hecke pairs of polynomial growth possess property (RD), see \cite{s4}.

We conclude this section with another application of the Schlichting completion of Hecke pairs. Recently, there has been some interest about property (T) of Hecke pairs , see for instance \cite{anan, larsen-palma} and their references. We refer to \cite{larsen-palma} for definition of property (T) of pairs of topological groups and basic results concerning property (T) of Hecke pairs. Here we observe that relative unimodularity of discrete Hecke pairs is a necessary condition for property (T):
\begin{proposition}
\label{prop:PTunimod}
    If a discrete Hecke pair $\gh$ has property (T), then it is relatively unimodular.
\end{proposition}
\begin{proof}
    Let $\ghb$ be the Schlichting completion of $\gh$. By Theorem A.8 of \cite{larsen-palma}, $\gb$ has property (T). Now it follows from Corollary 1.3.6 of \cite{BHV} that $\gb$ is unimodular. Thus $\gh$ is relatively unimodular.
\end{proof}


\section{When $H$ is a cocompact subgroup of $G$}
\label{sec:Hcocompact}

In this section, we assume that $H$ is a unimodular closed subgroup of $G$ such that the homogeneous space $H\ba G$ has a finite relatively invariant measure $\nu$. Therefore by Corollary B.1.8 of \cite{BHV}, $G$ is unimodular and the measure $\nu$ on $H\ba G$ is right $G$-invariant, and so both Assumptions \ref{assume:r-G-inv} and \ref{assume:H-unimod} hold for the pair $\gh$.

\begin{theorem}
\label{thm:cocompact} Given a pair $\gh$ as above, the left regular representation $\la:\H\gh\to B(L^2(\hg))$ is a well defined $\ast$-homomorphism.
\end{theorem}
\begin{proof}
For a given $f\in \H\gh$, we set $M_f:=\max\{ |f(Hx)| ; Hx\in \hg\}$. Then for every $\xi\in L^2(\hg)$, by applying Minkowski's inequality for integrals (Theorem 6.19 of \cite{folland-ra}), we obtain
\begin{eqnarray*}
\|f\ast \xi\|_2&=& \left(\int_{\hg} \left|\int_{\hg} f(Hxy\inv) \xi(Hy)d\nu(Hy)\right|^2 d\nu(Hx)\right)^{1/2}\\
&\leq& \int_{\hg} \left(\int_{\hg} |f(Hxy\inv)|^2 |\xi(Hy)|^2 d\nu(Hy)\right)^{1/2}d\nu(Hx)\\
&\leq& M_f \int_{\hg} \|\xi\|_2 d\nu(Hx)\\
&=& M_f \nu(\hg)\|\xi\|_2.
\end{eqnarray*}
This shows that $f\ast \xi\in L^2(\hg)$. A straightforward computation shows that $\la$ is a homomorphism. Finally, since $G$ is unimodular, $\la$ preserves the involution, by Remark \ref{rem:star-hom}.
\end{proof}
The following corollary follows from the above theorem and Remark \ref{rem:condition-b}(iii):
\begin{corollary}
\label{cor:Hcocompact}
If $H$ is a unimodular cocompact subgroup of a locally compact group $G$, then the pair $\gh$ is a Hecke pair.
\end{corollary}
We do not know whether pairs $\gh$ for which $\hg$ has a finite relatively invariant measure satisfy Condition (b) of Definition \ref{def:heckepair-allcases}(i). Regrading Theorem \ref{thm:cocompact}, this condition is the only thing that potentially prevents these pairs to be Hecke pairs.


\bibliographystyle{amsalpha}

\end{document}